\documentclass[12pt,leqno]{amsart}

\usepackage{graphicx}
\usepackage{pinlabel} %for algebraic and geometric topology
\usepackage{amsmath}
\usepackage{amsfonts}
\usepackage{amssymb}
\usepackage{mathtools}
\usepackage{amscd}
\usepackage[all,cmtip]{xy}
\usepackage{amsthm}
\usepackage{comment}
\usepackage{epsfig}
\usepackage[utf8]{inputenc}
\usepackage[colorinlistoftodos]{todonotes}
\usepackage[capitalise]{cleveref}
\usepackage{geometry}
 \usepackage{url} 
 
\makeatletter
\providecommand\@dotsep{5}
\def\listtodoname{List of Todos}
\def\listoftodos{\@starttoc{tdo}\listtodoname}
\makeatother

\newtheorem{theorem}{Theorem}[section]
\newtheorem{proposition}[theorem]{Proposition}

\newtheorem{lemma}[theorem]{Lemma}

  \theoremstyle{definition}
\newtheorem{definition}[theorem]{Definition}

\newtheorem{notation}[theorem]{Notation}
\newtheorem{assumption}[theorem]{Assumption}
\newtheorem{remark}[theorem]{Remark}

\newcommand{\dbZ}{\mathbb{Z}}

\newcommand{\Z}{\mathbb{Z}}

\newcommand{\calP}{{\mathcal P}}

\newcommand{\SUB}{\mathcal{ISO}}
\newcommand{\beq}{\begin{equation}}
\newcommand{\eeq}{\end{equation}}
\DeclareMathOperator{\mcd}{mcd}

\DeclareMathOperator{\rank}{rank}
\DeclareMathOperator{\Ind}{Ind}

\begin{document}

\title[]{On the group cohomology  of groups of the form $\dbZ^n\rtimes \dbZ/m$ with $m$ Square-free}

%\thanks{Support information for the second author.}

%    General info
\subjclass[2020]{Primary 20J06, 55T25, 20K25, 20H15}

%    Information for first author
\author{Luis Jorge S\'anchez Salda\~na}
\address{Departamento de Matem\' aticas, Facultad de Ciencias, Universidad Nacional Aut\'onoma de M\'exico, M\'exico D.F., Mexico}
\email{luisjorge@ciencias.unam.mx}
%    \thanks will become a 1st page footnote.
%\thanks{The author was supported by DGAPA-UNAM postdoctoral grant.}

%    Information for second author
\author{Mario Velásquez}
\address{Departamento de Matem\'aticas\\Facultad de Ciencias\\Universidad Nacional de Colombia, sede Bogot\'a\\Cra. 30 cll 45 - Ciudad Universitaria\\ Bogot\'a, Colombia}
\email{mavelasquezme@unal.edu.co}

\date{}
\dedicatory{The second author dedicates this paper to Lashmy.}
%\dedicatory{This paper is dedicated to our advisors.}

\keywords{Cohomology of groups, semidirect products, crystallographic groups, modules over group rings, modules with constant isotropy}

\begin{abstract}
We provide an explicit computation of the cohomology groups (with untwisted coefficients) of semidirect products of the form $\dbZ^n\rtimes \dbZ/m$ with $m$ square-free, by means of formulas that only depend on $n$, $m$ and the action of $\Z/m$ on $\Z^n$. We want to highlight the fact that we are not imposing any conditions on the $\Z/m$-action on $\Z^n$, and as far as we know, our formulas are the first in the literature in this generality. This generalizes previous computations of Davis-Lück and Adem-Ge-Pan-Petrosyan. In order to show that our formulas are usable, we develop a concrete example of the form $\dbZ^5\rtimes \dbZ/6$ where its cohomology groups are described in full detail.
\end{abstract}
\maketitle

\section{Introduction}

This paper provides a contribution to the problem of computing group cohomology of semidirect products of the form $\Z^n\rtimes G$ where $n\geq 1$, and $G$ is a finite group. We deal with the case where $G$ is a cyclic group of order a square-free positive integer $m$. 

Computations for general semidirect products $\Z^n\rtimes G$ can be difficult to obtain, mainly because the associated Lyndon-Hochschild-Serre spectral sequence does not collapse (see for example \cite{Totaro96}). Even  when $G$ is a finite cyclic group, there are examples where the Lyndon-Hochschild-Serre spectral sequence does not collapse; see, for example, \cite{Pan-Putricz} and \cite{LangerLuck}. On the other hand, provided $G$ is cyclic, there are some cases where the Lyndon-Hochschild-Serre spectral sequence collapses, for example, when $|G|$ is square-free (see \cite{AGPP}) or when the conjugation action of $G$ on $\Z^n$ is free outside of the origin (see \cite{LangerLuck}). In this paper we give a completely explicit computation of the group cohomology, with trivial coefficients, in the first case. In other words, we provide a computation of $H^l(\Z^n\rtimes \Z/m)$ in terms only of $l$, $m$, and the (complex eigenvalues of the) matrix determined by the action of a generator of $\Z/m$ on $\Z^n$.  In order to put our results in context it is worth saying that our computations generalize some results that already exist in the literature. For instance, Davis and Lück computed explicitly in \cite{DL13} the group cohomology of $\Z^n\rtimes \Z/p$, with $p$ prime, provided that the $\Z/p$-action on $\Z^n$ is free outside the origin. In Section \ref{un primo} we present a direct generalization of this computation dropping the hypothesis on the action. In \cite{AGPP} the authors work out an example of the form $\Z^3\rtimes \Z/4$.

Semidirect products $\Z^n\rtimes \Z/m$ have algebraic and geometric relevance. For instance, if the $\Z/m$ action of $\Z^n$ is faithful, then these groups are crystallographic. On the other hand,  $\Z/m$ acts on the torus $\mathbb{T}^n$,  then $\Z^n\rtimes\Z/m$ can also be identified with the orbifold fundamental group of $\mathbb{T}^n//(\Z/m)$ and then group cohomology of $\Z^n\rtimes\Z/m$ is related to the group of isomorphism classes of (flat) gerbes for these orbifolds; for details see \cite{AP}.

\begin{notation}\label{notation}
    Let $\{p_1,\dots ,p_\ell\}$ be a nonempty set of distinct prime numbers. Throughout the paper $G$ will denote the (finite cyclic) group $\Z/p_1 \times \cdots \times \Z/p_\ell$ of order $m=p_1\cdots p_\ell$. Given $p\in \{p_1,\dots ,p_\ell\}$, we denote by $G_p$ the quotient $G/(\Z/p)$.

    For the semidirect product $\Z^n\rtimes G$ the action of a fixed generator $g$ of $G$ on $\Z^n$ is given by an invertible $n\times n$ matrix with integer coefficients, which we will denote $\phi$. 
\end{notation}

For the rest of the introduction fix an integer $l\geq 0$. Our main goal is to compute the cohomology group $H^l(\Z^n\rtimes G)$. It is not difficult to see that these cohomology groups are finitely generated, and all torsion is $|G|$-torsion. Hence for  certain nonnegative integers $\theta(p_1),\dots, \theta(p_\ell)$ and $\omega$, that depend on $l$, we have an isomorphism

\[H^l(\Z^n\rtimes G)\cong \Z^\omega \oplus \Z/p_1^{\theta(p_1)}\oplus \cdots \oplus \Z/p_\ell^{\theta(p_\ell)}.\]

\vskip 10pt

From now on, fix $p\in \{p_1,\dots ,p_\ell\}$. Let us set some conventions and notation to state the computation of $\theta(p)$.

For the sake of readability, we state our main theorem under the assumption that $\Z^n$ splits, as a $\Z[G]$-module, as the following direct sum
\[\Z^r\oplus \Z[\Z/p]^s\oplus I^t\]
where $\Z^r$ is endowed with a trivial $\Z/p$ action, and $I^t$ is the augmentation ideal of $\Z[\Z/p]$. In this situation, we say that $\Z^n$ is an $(r,s,t)$-module. We also assume that $\phi$ splits as a block matrix $\phi_r\oplus \phi_s\oplus \phi_t$. In \cref{section:generaltorst} we explain why these assumptions are not restrictive in the sense that, in general, $\Z^n$ can be replaced by an $(r,s,t)$-module and $\phi$ always split.

In order to state the formula of the $p$-torsion, we set some preliminary notation.

\begin{itemize}    
    \item Define $\mathcal D$ as the set of positive divisors  $d$  of $m/p$ such that there is a primitive $\frac{m}{pd}$-th root of unity among the complex eigenvalues $\phi_t$ considered as a matrix with complex coefficients.
    
    \item For each $d\in \mathcal D$ define $\mathsf k_d$ as follows. Let $m_d$ is the number of eigenvalues of $\phi_t$ that are primitive $\frac{m}{pd}$-th roots of unity. Set $\mathsf k_d=m_d/(p-1)$. Note that $\mathsf k_d$ is a nonnegative integer due to \cite[Lemma 1.9(i)]{DL13}.

    \item Let $d\in \mathcal D$ and $i$ a nonnegative integer. Define $$\mathsf P_d^i=\sum_{j=0}^{\mathsf k_d} (-1)^j\binom{\mathsf k_d}{j}\binom{i-jp+\mathsf k_d-1}{\mathsf k_d-1}. $$
    Here we use the convention that $\binom{a}{b}=0$ whenever $a<0$ or $b<0$.
    
    \item Let $A\subseteq\mathcal D$, and $\beta$ a nonnegative integer. Denote by $\mcd(A)$ the maximum common divisor of the elements of $A$. Define $$\mathsf T(A,\beta)=\begin{cases}\left( \frac{p\mcd(A)}{m} \right)^{|A|} \sum_{\{(i_d)_{d\in \mathcal D}\mid i_d\geq 0, \sum i_d\leq \beta\}}\prod_{d\in A}( \mathsf P^{i_d}_d-1)& \text {if } A\neq\emptyset\\1& \text{ if } A=\emptyset\end{cases}$$

    \item Let $M$ be a $\Z[G]$-lattice. Choose a nonnegative integer $l$ and a positive divisor $d$ of $m$.
    Let us define an integer $\mathsf H(l,d,M)$ as follows. If $l>\rank_{\Z}(M)$, $\mathsf H(l,d,M)=0$. 
    
    Fix $l\leq \rank_{\Z}(M)$. Let $\mathbf{x}$ be the vector of all complex eigenvalues, including multiplicities, of a generator $g$ of $G$. Now let $\wedge^l \mathbf{x}$ be the vector of complex eigenvalues of the $l$-th exterior power of $g$, that is $\wedge^l \mathbf{x}$ is the vector whose coordinates are all products of collections of $l$ elements of coordinates of $\mathbf{x}$, in particular this vector has $\binom{rank_{\Z}(M)}{l}$ coordinates. By definition $\mathsf H(l,d,M)$ is the number of coordinates of $\wedge^l \mathbf{x}$ that are $\frac{m}{d}$-th roots of unity.
\end{itemize}

The following theorem is essentially \cref{main-thm} after applying the arguments presented in \cref{section:from:main:to:intro}.

\begin{theorem}\label{main-thm:intro}
For any $l$ the $p$-torsion part of $H^l(\Z^n\rtimes G)$ is isomorphic to
\begin{itemize}
\item For $l$ even: $(\Z/p)^\theta$,  where $$\theta=\sum_{\substack{l_1+l_2=l\\ l_2\text{ even}}}\sum_{A\subseteq\mathcal D}\left(\sum_{\tau=0}^s\mathsf T(A,l_2-p\tau)\binom{s}{\tau}\right)\cdot \mathsf H(l_1, \mcd(A),\Z^r)$$
    \item For $l$ odd:
    $(\Z/p)^\theta$,  where $$\theta=\sum_{\substack{l_1+l_2=l\\ l_2\text{ odd}}}\sum_{A\subseteq\mathcal D}\left(\sum_{\tau=0}^s\mathsf T(A,l_2-p\tau)\binom{s}{\tau}\right)\cdot \mathsf H(l_1, \mcd(A),\Z^r)$$
\end{itemize}
\end{theorem}

Now we state the computation of the rank of $H^l(\Z^n\rtimes G)$. It is worth saying that the computation of the rank is not difficult, and it is done in \cref{section:rank}.

\begin{theorem}\label{intro:rank}
    For each $l\geq 0$, we have
    \[\rank_{\Z}(H^l(\Z^n\rtimes G))= \mathsf H(l, m, \Z^n). \]
\end{theorem}

\subsection{Outline of the paper}

In \cref{section:useful:results} we recall some results of Davis-Lück and Langer-Lück about the computation in the case where the action of $G$ on $\Z^n$ is free outside the origin. Our computations rely strongly on this particular case. The main result of \cref{section:reduction:to:freeoutorigin} is \cref{cor:torsion:st}, in this proposition essentially we state that the computation of the $p$-torsion of $H^l(\Z^n\rtimes G)$ can be roughly reduced to the computation of the $p$-torsion of $H^*(I^t\rtimes\Z/p)$. Since $\Z/p$ acts freely outside the origin on $I$, the results of the previous section become useful. Although the torsion subgroup of $H^\beta(I^t\rtimes\Z/p)$ was already computed by Davis-Lück, we re-compute this torsion subgroup together with its structure of $\Z[G_p]$-module since this is crucial for our purposes; this is done in the main result of \cref{section:ptorsion:It}, namely \cref{permutation-module}. \cref{section:ptorsion}  is a one-theorem section. In \cref{main-thm} we provide a computation for the $p$-torsion of $H^l(\Z^n\rtimes G)$ in terms of the results obtained in Sections 3 and 4. In \cref{section:from:main:to:intro} we explain how to state theorem \cref{main-thm} in terms of the constants defined in this introduction, so that we recover \cref{main-thm:intro}. \cref{section:rank} is a short section devoted to prove \cref{intro:rank}. In \cref{section:generaltorst} we explain how to replace $\Z^n$ with a $\Z[G]$-module of type $(r,s,t)$. Also in this section, we explain how to compute algorithmically $r$, $t$, $\phi_r$ and $\phi_t$ out of the $\Z[G]$-module structure of $\Z^n$. Finally, in \cref{un primo} we explain how \cref{main-thm:intro} and \cref{intro:rank} reduce in the case $G$ is a cyclic group of prime order.

We also have \cref{section:rst}, where we prove some splitting results for $\Z[G]$-lattices that might be well known for the experts on representation theory of finite cyclic groups, but we were not able to find them in the literature.

\subsection*{Acknowledgements.}
We are grateful for the financial support of DGAPA-UNAM grant PAPIIT IA106923. L.J.S.S thanks the Universidad Nacional de Colombia, sede Bogot\'a, for its hospitality during a visit in an early stage of this project. The authors thank Omar Antol\'in-Camarena for pointing out an error in an early stage of this paper.

\section{Some useful results}\label{section:useful:results}

Let $H$ be a finite group. Whenever $H$ acts freely on $\Z^n-\{0\}$, we say $H$ acts \emph{free outside the origin} on $\Z^n$. For $H=\Z/q$, with $q$ prime,  acting freely outside the origin on $\Z^n$, there is an explicit computation of the group cohomology, with nontwisted integral coefficients, of $\Z^n\rtimes\Z/q$ due to Davis and L\"uck in \cite{DL13}. The following result will be useful for our purposes. Let us fix a generator $\xi\in \Z/q$ and denote by $\rho:\Z^n\to\Z^n $ the action of $\xi$.

\begin{lemma}\label{DL1} Let $\mathcal{P}$ be the set of conjugacy classes of finite maximal subgroups of $\Z^n\rtimes\Z/q$. Suppose $\Z/q$ acts freely outside the origin on $\Z^n$. We have the following
\begin{enumerate}
\item We have an isomorphism of $\Z[\Z/q]$-modules, 
    $$\Z^n \cong I_1\oplus\cdots \oplus I_k,$$ where the $I_j$ are non-zero ideals of $\Z[\Z/q]/(N)$, and $N$ is the norm element. The number $k$ satisfies $n=(q-1)k.$ 
    \item There is an isomorphism $H^1(\Z/q;\Z^n)\cong (\Z/q)^k$ and the latter is in bijective correspondence with $\mathcal{P}$. The bijection sends an element $\bar{x}$ of $\Z^n/(\rho-id)\Z^n$ to the subgroup generated by $(\bar{x},\xi)$. 
    \item The map induced by the restrictions to the various subgroups $$H^\beta(\Z^n\rtimes\Z/p)\to H^\beta(\Z^n)^{\Z/q}\oplus\bigoplus_{(P)\in\mathcal{P}} H^\beta(P)$$ is injective and is bijective if $\beta>n$.
    \end{enumerate}
\end{lemma}
\begin{proof}
\begin{enumerate}
    \item Lemma 1.9 (i) in \cite{DL13}.
    \item Lemma 1.9 (iii) in \cite{DL13}.
    \item Lemma 1.7 (ii), (iii) and Remark 1.8 in \cite{DL13}.
    \end{enumerate}
\end{proof}

We also need the following result from \cite{LL12}. 

\begin{lemma}\cite[Ex. 7.7]{LL12}
\label{LL1}
Let $G$ and $p$ be as in \cref{notation}.
 Assume $M$ is a finitely generated free abelian group endowed with an action of $G$ free outside the origin, let $C\subseteq G$ be a subgroup, and let $\mathcal{M}(C)$ be the set of conjugacy classes of finite maximal subgroups in $M\rtimes G$ isomorphic to $C$.  Then
$$|\mathcal{M}(C)|=\begin{cases}
    1& \text{ if } C=G;\\
    0& \text{ if } C\neq G \text{ or $C$ is not a $p$-Sylow subgroup};\\
    \frac{p(p^{\rank(M)/(p-1)}-1)}{m}& \text{ if } C \text{ is a $p$-Sylow subgroup}.
\end{cases}$$
\end{lemma}

\section{Reducing the computation of the $p$-torsion of $H^*(\Z^n\rtimes G)$ to actions free outside the origin}\label{section:reduction:to:freeoutorigin}

In this section we consider $G$, $p$, and $G_p$ as in \cref{notation}. \\

Denote $$H^*(-)_{(p)}=H^*(-)\otimes_\Z \Z_{(p)}.$$ Note that $H^*(-)_{(p)}$ is a cohomology theory. 

\begin{remark}
    Tensoring with $\Z_{(p)}$ kills all torsion that is not $p$-torsion. Of particular interest for us is the fact that the $p$-torsion subgroups of $H^*(\Z^n\rtimes G)_{(p)}$ and of $H^*(\Z^n\rtimes G)$ are isomorphic. The main goal of this section is to compute the $p$-torsion of the former cohomology group.
\end{remark}

\begin{lemma}[Reducing to one prime]\label{lemma:reducing:to:invariants}
    For all  $l\geq0$
 \begin{equation}\label{LHS-collapse}
H^l(\Z^n\rtimes G)_{(p)}\cong \left(H^l(\Z^n\rtimes\Z/p)_{(p)}\right)^{G_p},
 \end{equation}
 where the $G_p$ action on $H^l(\Z^n\rtimes\Z/p)_{(p)}$ is the one induced by conjugation of $G_p$ on $\Z^n\rtimes\Z/p$.
\end{lemma}
\begin{proof}
Consider the group extension
\begin{equation}\label{ext}
    0\to\Z^n\rtimes\Z/p\to \Z^n\rtimes G\to G_p\to 0.
\end{equation}

From the Lyndon-Hochschild-Serre spectral sequence for $H^*(-)_{(p)}$ associated to the extension (\ref{ext}), we obtain the following 
\begin{equation*}%\label{ssv2}
E_2^{j,l}=H^j(G_p;H^l(\Z^n\rtimes\Z/p)_{(p)})=\begin{cases}\left(H^l(\Z^n\rtimes\Z/p)_{(p)}\right)^{G_p} & j=0\\
 0&j\neq0.\end{cases}
\end{equation*}
The second case in the above equation holds since, for any  $G_p$-module $M$, the group $H^l(G_p;M)$ has exponent $|G_p|$ for every $l>0$, in particular, it has no $p$-torsion.
Now the proof follows.
\end{proof}

\vskip 10pt
From \Cref{lemma:reducing:to:invariants}, it is clear that our next task is to compute $\left(H^l(\Z^n\rtimes\Z/p)_{(p)}\right)^{G_p}$ for $l\geq 0$. The cohomology of groups of the form $\Z^n\rtimes\Z/p$ has been widely studied in \cite{AGPP} and \cite{DL13}, and we will borrow some ideas from these references.
In particular, we will start to make use of the theory of modules of type $(r,s,t)$ from \cref{section:rst}.

\vskip 10pt

We say a finitely generated $\Z[G]$-module $M$ is a \emph{$\Z[G]$-lattice} provided $M$ is free as a $\Z$-module.

\begin{assumption}  
From now on we will suppose that $\Z^n$, considered as a $\Z[\Z/p]$-lattice, is endowed with an  $(r,s,t)$-decomposition that also is $G$-equivariant. This assumption is not restrictive due to \cref{lemma:rst} and \cref{prop:virtual:rst:decompositions}.
\end{assumption}

\vskip 10pt

Eventually, in our calculation of $\left(H^l(\Z^n\rtimes\Z/p)_{(p)}\right)^{G_p}$ we will use the Lyndon-Hochshild-Serre spectral sequence associated to $\Z^n\rtimes\Z/p$, thus it will be crucial to understand the $\Z/p$-action, and actually the $G$-action, on the cohomology groups of $\Z^n$.

\begin{lemma}\label{lemma:kunneth} Let $M$ be a $\Z[G]$-module such that $M$, considered as a $\Z[\Z/p]$-module, is of type $(r,s,t)$ with associated $G$-equivariant decomposition $\Z^r\oplus \Z[\Z/p]^s\oplus I^t$. Then the action of $\Z/p$ on $I^t$ is free outside of the origin and, for all $\omega\geq0$, we have the following  isomorphism of $\Z[G]$-modules
\begin{align}\label{kunneth}
H^\omega(M)&\cong\bigoplus_{\alpha+\delta+\gamma=\omega} H^\alpha(\Z^r)\otimes_{\Z}H^\delta(\Z[\Z/p]^s)\otimes_{\Z}H^\gamma(I^t).
\end{align}

Moreover

\[H^\delta(\Z[\Z/p]^s)\otimes_{\Z}H^\gamma(I^t)\cong \begin{cases} H^\gamma(I^t)^{\binom{s}{\tau}}\oplus \Ind_{0}^{\Z/p}(H^\gamma(I^t))^{s_\delta} &\text{if }\delta=p\tau, \tau\leq s\\ \Ind_{0}^{\Z/p}(H^\gamma(I^t))^{s_\delta}&\text{if }\delta\neq p\tau, \tau\leq s   
\end{cases}
\]
Where $s_\delta$ is the rank of the free factor of $H^\delta(\Z[\Z/p]^s)$ as $\Z[\Z/p]$-module, that is $$s_\delta=\frac{1}{p}\sum_{(i_1,\ldots ,i_s)\in I_\delta}\prod_{k=1}^s\binom{p}{i_k},$$ with $I_\delta=\{(i_1,\ldots ,i_s)\mid i_1+\cdots+i_s=\delta, i_k\notin\{0,p\} \text{ for some $k=1,\ldots,s$} \}$.
\end{lemma}
\begin{proof}
 The Künneth formula leads to the isomorphism (\ref{kunneth}). 
Note that $\Z/p$ acts free outside the origin on the augmentation ideal $I$ and so does on $I^t$.
On the other hand, in \cite[p. 382]{DL13} it is proved that $H^*(\Z[\Z/p])$ is free as a $\Z/p$-module for $1\leq * \leq p-1$ and isomorphic to $\Z$ for $*=0, p$. Then $$H^\delta(\Z[\Z/p]^s)\cong \begin{cases}\Z^{\binom{s}{\tau}}\oplus \Z[\Z/p]^{s_\delta} &\delta=p\tau, \tau\leq s\\\Z[\Z/p]^{s_\delta}& \delta\neq p\tau, \tau\leq s\end{cases}$$ Now, the moreover part follows from the previous observations and \cite[Corollary 5.7]{BrownBook}.
\end{proof}

The following proposition is a direct consequence of the above lemma and the Shapiro lemma.

\begin{proposition}\label{prop:5.6} We have, for all $\beta\geq 0$ and $\omega \geq 0$, the following isomorphisms of $\Z[G]$-modules

    \[H^\beta(\Z/p;H^\omega(\Z^n))\cong \bigoplus_{\alpha +\delta +\gamma=\omega} H^\alpha(\Z^r) \otimes H^\beta (\Z/p;H^\delta(\Z[\Z/p]^s)\otimes_{\Z}H^\gamma(I^t)).\]

and

$$H^\beta(\Z/p; H^\delta(\Z[\Z/p]^s)\otimes_\Z H^\gamma(I^t))\cong \begin{cases}
    H^\beta(\Z/p;H^\gamma(I^t)^{\binom{s}{\tau}})\oplus H^\delta(I^t)^{s_\delta}&\text{ if } \delta=p\tau, \tau\leq s\\
    H^\delta(I^t)^{s_\delta}&\text{ if }\delta\neq p\tau, \tau\leq s
\end{cases}$$

\end{proposition}

\begin{proposition}\label{cor:torsion:st}
The torsion subgroup of $H^l((\Z[\Z/p]^s\oplus I^t)\rtimes\Z/p)$ is isomorphic as $\Z[G]$-modules to the torsion subgroup of
\begin{align}\label{torsion}
\bigoplus_{\tau=0}^sH^{l-p\tau}(I^t\rtimes\Z/p)^{\binom{s}{\tau}}
\end{align}    
\end{proposition}
\begin{proof} As a direct application of the Lyndon-Hochshild-Serre spectral sequence, the fact that it collapses for the extension $(\Z[\Z/p]^s\oplus I^t)\rtimes\Z/p$ \cite[Theorem 1.1]{AGPP}, and the Künneth formula we have the following isomorphisms

\begin{align*}
    H^\beta ((\Z[\Z/p]^s\oplus I^t)\rtimes\Z/p) &\cong \bigoplus_{\alpha + \theta = \beta} H^\alpha (\Z/p; H^\theta (\Z[\Z/p]^s\oplus I^t))\\
    &\cong \bigoplus_{\alpha + \theta = \beta } H^\alpha (\Z/p ; \bigoplus_{\delta + \gamma=\theta} H^\delta(\Z[\Z/p]^s)\otimes H^\gamma(I^t))\\
    &\cong \bigoplus_{\alpha +\theta =\beta }\quad\bigoplus_{\delta +\gamma=\theta} H^\alpha (\Z/p ;  H^\delta(\Z[\Z/p]^s)\otimes H^\gamma(I^t))
\end{align*}

Recall that the cohomology of $\Z/p$ is all torsion except possibly in degree 0. Hence we get, from the above isomorphisms, that the torsion subgroup of $H^\beta ((\Z[\Z/p]^s\oplus I^t)\rtimes\Z/p)$ is isomorphic to

\begin{align*}          \bigoplus_{\substack{\alpha +\theta =\beta\\ \alpha>0}} \quad\bigoplus_{\delta +\gamma=\theta} H^\alpha (\Z/p ;  H^\delta(\Z[\Z/p]^s)\otimes H^\gamma(I^t)) 
\end{align*}
which, by \cref{prop:5.6} is isomorphic to 
\begin{align*}
\bigoplus_{\substack{\alpha +\theta =\beta\\ \alpha>0}} \quad\bigoplus_{\substack{p\tau+\gamma=\theta\\ \tau\leq s}} H^\alpha (\Z/p;H^\gamma(I^t)^{\binom{s}{\tau}}) &\cong  \bigoplus_{\substack{\alpha +p\tau +\gamma  =\beta\\ \alpha>0,\tau\leq s}} H^\alpha (\Z/p;H^\gamma(I^t)^{\binom{s}{\tau}}) \\
&\cong \bigoplus_{\tau=0}^s \bigoplus_{\alpha +p\tau +\gamma =\beta } \mathrm{torsion}\left( H^\alpha (\Z/p;H^\gamma(I^t)^{\binom{s}{\tau}}) \right)
\end{align*}

As a conclusion we have

\begin{equation}\label{eq:torsion1}\mathrm{torsion}\left( H^\beta ((\Z[\Z/p]^s\oplus I^t)\rtimes\Z/p) \right) \cong \bigoplus_{\tau=0}^s \bigoplus_{\alpha +p\tau +\gamma =\beta } \mathrm{torsion}\left( H^\alpha (\Z/p;H^\gamma(I^t)^{\binom{s}{\tau}}) \right)
\end{equation}

On the other hand we have the following isomorphisms, where the second one is a direct application of the Lyndon-Hochschild-Serre spectral sequence and the fact that it collapses for the extension $ I^t\rtimes\Z/p$ \cite[Theorem 1.1]{AGPP}

\begin{align*}
    \mathrm{torsion}\left( \bigoplus_{\tau=0}^s H^{\beta -p\tau} (I^t\rtimes \Z/p)^{\binom{s}{\tau}}  \right) &\cong \bigoplus_{\tau=0}^s \mathrm{torsion}\left(  H^{\beta -p\tau} (I^t\rtimes \Z/p)^{\binom{s}{\tau}}  \right)\\
    & \cong  \bigoplus_{\tau=0}^s \mathrm{torsion}\left(\left( \bigoplus_{\alpha' +\theta'=\beta -p\tau} H^{\alpha'} (\Z/p; H^{\theta'}(I^t) )\right)^{\binom{s}{\tau}}  \right)\\
    &\cong  \bigoplus_{\tau=0}^s \mathrm{torsion}\left(\left( \bigoplus_{\alpha' +\theta'=\beta -p\tau} H^{\alpha'} (\Z/p; H^{\theta'}(I^t)^{\binom{s}{\tau}} )\right)  \right)
\end{align*}
Now the claim follows combining the last chain of isomorphisms and \cref{eq:torsion1}.
\end{proof}

\section{The $p$-torsion of $(H^*(I^t\rtimes \Z/p)_{(p)})^{G_p}$}\label{section:ptorsion:It}

Let $L$ be a $\Z[G]$-lattice such that  $\Z/p$ acts freely outside the origin on $L$. Denote by $\mathcal{P}_L$ the set of conjugacy classes of subgroups of order $p$ in $L\rtimes\Z/p$.

\begin{lemma}\label{Lemma:luck:equivariant:embedding}
The map in \cref{DL1} (3), applied to $L\rtimes \Z/p$, is $G_p$-equivariant; in particular, we have an injective map (bijective for the torsion free part or when $l>\rank(L)$)
\begin{equation}\label{iso1}H^l(L\rtimes\Z/p)^{G_p}\to H^l(L)^{G_p}\oplus\left(\bigoplus_{(P)\in\mathcal{P}_L} H^l(P)\right)^{G_p}.\end{equation}
Moreover
\begin{equation}\label{iso2}\left(\bigoplus_{(P)\in\mathcal{P}_L} H^l(P)\right)^{G_p}\cong H^l(P)^{|\mathcal{P}_L/G_p|}\end{equation}
 where the right hand side of the isomorphism stands for a direct sum of $|\mathcal{P}_L/G_p|$ copies of $H^l(P)$. 
\end{lemma}

\begin{proof}
    Consider the conjugation $G_p$-action on $L\rtimes G$, and note that this action restricts to the subgroup $L$, and  it induces a permutation action on $\mathcal{P}_L$.
Now as the $\Z/p$-action on $L$ is free outside the origin, \cref{DL1} applies. As the map in \cref{DL1} is induced by restrictions, and $H^l(L)$ and $\bigoplus_{(P)\in\mathcal{P}_L} H^l(P)$ have an induced $G_p$-action, we can see now that such a map is $G_p$-equivariant. The first claim follows after taking $G_p$-invariants.

Since  every finite subgroup of $L\rtimes G$ is abelian, then any element of $G_p$ that normalizes a subgroup of order $p$ in $L\rtimes G$ centralizes such a finite subgroup. Hence  the $G_p$-module $\bigoplus_{(P)\in\mathcal{P}_L} H^l(P)$ is a permutation module. Now the \textit{moreover part} follows.
\end{proof}

Denote $I^t$ by $N$ and $\mathcal{P}_{I^t}$ by $\mathcal{P}$. 
Since the action of $\Z/p$ on $N$ is free outside the origin, we can apply \Cref{cor:decomposition:constant:isotropy:v3} to obtain a finite index submodule $\bigoplus_{H\leq G_p}N_H$ of $N$ where every non-zero element of $N_H$ has  isotropy $H$.
 In particular $G/H$ acts free outside the origin on $N_H$.
 
 Let $\mathcal{P}_H$ be the set of conjugacy classes of subgroups of order $p$ in $N_H\rtimes\Z/p$.
 
\begin{lemma}\label{lemma:P_j:product}
    There is an isomorphism of $G_p$-sets 
    \begin{align*}
        \prod_{H\leq G_p}\mathcal{P}_H\cong\mathcal{P}.
    \end{align*}

    Moreover $|\mathcal{P}_H|=p^{k_H}$, where $ k_H=\rank_{\Z}(N_H)/(p-1)$.
\end{lemma}
\begin{proof}
    By \cref{cor:decomposition:constant:isotropy:v3} we have the following inclusion of $G$-modules (in particular of $\Z/p$-modules)
    \begin{equation}\label{eq1:G-sets}\bigoplus_{H\leq G_p}N_H\leq N\end{equation} %$M_\xi\to K_{\{p_j\}}^\perp$
    that is an isomorphism after tensoring with $\Z_{(p)}$.

    By Lemma \ref{DL1} (2) we have, we have $H^1(\Z/p;N_H)\cong (\Z/p)^{k_H}$ for certain integer $k_H$. Hence $H^1(\Z/p;N_H)\otimes \Z_{(p)} \cong H^1(\Z/p;N_H\otimes \Z_{(p)}) \cong (\Z/p)^{k_H}$, and the latter group is in bijection with $\mathcal P_H$. On the other hand $H^1(\Z/p;\bigoplus_{H\leq G_p} N_H \otimes \Z_{(p)}) \cong H^1(\Z/p;N\otimes\Z_{(p)} )$ by \eqref{eq1:G-sets}.
    
    The observations in previous paragraph lead to
\begin{align*}
    \mathcal{P} &\cong H^1(\Z/p;N)\\
    &\cong H^1(\Z/p;N\otimes\Z_{(p)})\\
    &\cong H^1(\Z/p;\bigoplus_{H\leq G_p} N_H\otimes\Z_{(p)})\\
    &\cong H^1(\Z/p;\bigoplus_{H\leq G_p} N_H)\\
    &\cong \bigoplus_{H\leq G_p}H^1(\Z/p;N_H)\\
    &\cong \prod_{H\leq G_p}\mathcal{P}_H
\end{align*}
    From the description of the isomorphism in Lemma \ref{DL1} (2), it follows that all of the above isomorphisms are compatible with the $G_p$-action. Therefore we have the desired result.

    The moreover part follows immediately from  \cref{DL1} (1) and (2).
\end{proof}

 Let $\SUB_{G_p}(N)$ the collection of subgroups of $G_p$ appearing as isotropy groups of non-zero elements in $N$. Consider a subset $A$ of $\SUB_{G_p}(N)$. Let us denote by $\cap A$ the intersection of the elements of $A$, and by convention when $A$ is empty we set $\cap A=G_p$. The following lemma will not be used latter, still the arguments within the proof will be used in the proof of \cref{permutation-module}. 

\begin{lemma}[Splitting $\mathcal{P}$ in $G_p$-orbits]\label{lemma: P_j:orbits} Let $A\subseteq\SUB_{G_p}(N)$. For $A=\emptyset$ define $\theta(A)=1$, otherwise define 
\begin{equation}\label{conteo:orbitas} \theta(A)=\left(\frac{|\cap A|}{|G_p|}\right)^{|A|}\prod_{H\in A}(p^{k_H}-1),  \text{ and } k_H=\rank_{\Z}(N_H)/(p-1). 
\end{equation}
We have an isomorphism of $G_p$-sets\begin{equation}\label{decomp-orbits}
\mathcal{P}\cong  \bigsqcup_{ A\subseteq\SUB_{G_p}(N)} \bigsqcup_{i=1}^{\theta(A)}  (G_p/\cap A)
\end{equation}
\end{lemma}
\begin{proof}
For each subgroup $H$ of $G$ we have the following decomposition of $G_p$-sets $$\mathcal{P}_H=(\mathcal{P}_H-\{0\}_H)\sqcup\{0\}_H,$$ where $\{0\}_H$ denotes the $G_p$-conjugacy class of $\{0\}\rtimes\Z/p$ in $N_H\rtimes\Z/p$, in particular $\{0\}_H$ is a singleton.

Denote by $\mathcal{G}_p$ the collection of all subgroups of $G_p$. By \Cref{lemma:P_j:product} and the distributivity property of the cartesian product over the disjoint union we have the following isomorphisms of $G_p$-sets
\begin{align*}
    \mathcal{P}&\cong \prod_{H\leq G_p}\mathcal{P}_H\\
    &\cong \bigsqcup_{ \mathcal{I}\subseteq   \mathcal{G}_p}\left( \prod_{H\in \mathcal{I}}  (\mathcal{P}_H-\{0\}_H)\times \prod_{K\in\mathcal{G}_p- \mathcal{I}} \{0\}_K \right),\\
    &\cong \left(\bigsqcup_{ \emptyset\neq\mathcal{I}\subseteq   \mathcal{G}_p} \prod_{H\in \mathcal{I}}  (\mathcal{P}_H-\{0\}_H)\right)\sqcup \{0\}_{G_p}\\
    &\cong \left(\bigsqcup_{ A\subseteq   \SUB_{G_p}(N) } \prod_{H\in A}  (\mathcal{P}_H-\{0\}_H)\right)\sqcup\{0\}_{G_p}
\end{align*}
where the third isomorphism follows from the fact that each $\{0\}_K$ is a singleton, and the fourth isomorphism follows from the fact that $\mathcal{P}_H-\{0\}_H$ is empty provided that $H$ is not an isotropy group of the module $N$.

Next we prove that every element in  $\mathcal{P}_H-\{0\}_H$ has isotropy $H$.  Note that $$N_H\rtimes (G/H)\cong (N_H\rtimes \Z/p)\rtimes G_p/H,$$ and that every subgroup of order $p$ in the latter double semidirect product lies in $N_H\rtimes \Z/p$.  Thus there is a canonical identification between $\mathcal{P}_H/G=\mathcal{P}_H/G_p=\mathcal{P}_H/(G/(H\times \Z/p))$ and the set of conjugacy classes of subgroups of order $p$ in the group $N_H\rtimes (G/H)$. Note that in the latter semidirect product $G/H$ acts free out of the origin in $N_H$. We claim that the stabilizer in $G_p/H$ of every element of $\mathcal{P}_H-\{0\}_H$ is trivial. Let us prove the claim, assume that $x\in G_p/H$ stabilizes $(P)$, then $xz$ normalizes $P$ for some $z\in \Z/p$. Hence $\langle P,xz \rangle$ is a finite subgroup of $N_H\rtimes G/H$ that contains properly $P$, thus by \cref{LL1} is conjugate to $G_p/H$. This is a contradiction since $(P)\neq \{0\}_H$. We conclude every element in $\mathcal{P}_H-\{0\}_H$ has isotropy $H$.

Now it is easy to conclude that every element in $\prod_{H\in A}  (\mathcal{P}_H-\{0\}_H)$ has isotropy $\bigcap A$. As a consequence all of its orbits have  $|G_p|/|\bigcap A|$ elements. It only remains to count the number of $G_p$-orbits in $\prod_{H\in A}  (\mathcal{P}_H-\{0\}_H)$. By \cref{lemma:P_j:product} we have that $|\mathcal{P}_H-\{0\}_H|=p^{k_H}-1$, and therefore the number of orbits is
\[\frac{|\bigcap A|\prod_{H\in A} (p^{k_H}-1) }{|G_p|}.\]
\end{proof}

Let $H\leq G_p$ be a subgroup and $\beta$ be a nonnegative integer. Recall that $k_H = \rank(N_H)/(p-1)$, see \eqref{conteo:orbitas}. Also recall that there is a bijection between $\mathcal{P}_H$ and $(\Z/p)^{k_H}$. Consider the following sets
\begin{itemize}\item$\mathcal{P}_H^i=\{([x_1],\ldots,[x_{k_H}])\in(\Z/p)^{k_H}\mid 0\leq x_j \leq p-1 \text{ and }  \sum_jx_j=i\}.$
\item $\mathcal{P}_H^{\leq\beta}=\bigsqcup_{i\leq\beta}\mathcal{P}_H^i.$
\end{itemize}
 We are regarding both $\mathcal{P}_H^{i}$ and $\mathcal{P}_H^{\leq\beta}$ as subsets of $\mathcal{P}$ via the isomorphism with $(\Z/p)^{k_H}$. Since $\calP$ also is in bijection with $(\Z/p)^k$ with $k=\rank_\Z(N)/(p-1)$, we have an analogous definition for  $\mathcal{P}^{i}$ and  $\mathcal{P}^{\leq\beta}$. 

 \begin{lemma}\label{lemma:image:luck:map} Let $H\leq G_p$, and
    let
    $$H^\beta(N_H\rtimes\Z/p)\to H^\beta(N_H)^{\Z/p}\oplus\bigoplus_{(P)\in\mathcal{P}_H} H^\beta(P)$$
    be the homomorphism given in \Cref{DL1}(3). Then the image of the torsion subgroup of $H^\beta(N_H\rtimes\Z/p)$ is equal to $$\bigoplus_{(P)\in\mathcal{P}_H^{\leq \beta}} H^\beta(P).$$ 
 
     In particular, for all $i,\beta\geq 0$, $\mathcal{P}_H^i$ and 
$\mathcal{P}_H^{\leq\beta}$ are $G_p$-subsets of $\calP_H$.
 \end{lemma}
 \begin{proof}
    The claim about the image of the torsion subgroup follows from Lemma 1.10 (i) in \cite{DL13}. By \cref{Lemma:luck:equivariant:embedding},  $\bigoplus_{(P)\in\mathcal{P}_H} H^\beta(P)$ is a permutation module, and then $\mathcal{P}_H^{\leq\beta}$ has to be a $G_p$-set for any $\beta$. Since $\mathcal{P}_H^\beta=\mathcal{P}_H^{\leq \beta}-\mathcal{P}_H^{\leq \beta-1}$, we conclude it is also a $G_p$-set.
 \end{proof}

The torsion subgroup of $H^\beta(I^t\rtimes\Z/p)$ was already computed by Davis and Lück in Theorem 1.7 (i) from \cite{DL13}. Nevertheless, we re-compute this torsion subgroup together with its structure of $\Z[G_p]$-module since this is crucial for our purposes.

\begin{proposition}[Computing the $p$-torsion of $H^\beta(I^t\rtimes\Z/p)$ as $G_p$-module]\label{permutation-module}
Let us denote by $\mathcal{A}^\beta_p$ the $p$-torsion part of $H^\beta(I^t\rtimes\Z/p)$.  Then, we have an isomorphism of $G_p$-modules
\begin{equation}\label{perm-module}
\mathcal{A}^\beta_p\cong \bigoplus_{\mathcal{P}^{\leq\beta}} \Z/p\cong\bigoplus_{A\subseteq\SUB_{G_p}(I^t)}(\Z/p[G_p/\cap A])^{\theta(A,\beta)},
\end{equation}
where $$\theta(A,\beta)=\begin{cases}\left( \frac{|\cap A|}{|G_p|} \right)^{|A|} \sum_{\{(i_H)_{H\in \SUB_{G_p}(I^t)}\mid i_H\geq 0, \sum i_H\leq \beta\}}\prod_{H\in A}(\left|\mathcal P^{i_H}_H\right|-1)& \text { if } A\neq\emptyset\\1& \text{ if } A=\emptyset\end{cases}$$
whenever $\beta$ is even and $\theta(A,\beta)=0$ whenever $\beta$ is odd.
\end{proposition}
\begin{proof}
    By \cref{lemma:image:luck:map} we have that $$\mathcal{A}_p^\beta \cong \bigoplus_{(P)\in\mathcal{P}^{\leq \beta}} H^\beta(P)\cong\begin{cases}
        \bigoplus_{\mathcal{P}^{\leq\beta}} \Z/p & \beta \text{ even,}\\
        0 &\beta \text{ odd.}
    \end{cases} $$ and this establishes the first isomorphism in \eqref{perm-module}. Recall that $\bigoplus_{\mathcal{P}^{\leq\beta}}H^\beta(P)$ is a permutation $\Z[G_p]$-module. The next natural step will be  to describe $\mathcal{P}^{\leq \beta}$ in terms of its $G_p$-orbits, and such a description will lead to the second isomorphism in \eqref{perm-module}.

 By \Cref{lemma:P_j:product}, we have that $\mathcal P$ is in correspondence with $\prod_{H\in\SUB_{G_p}(I^t)}\mathcal{P}_H$, thus every tuple in the first set can be decomposed as a tuple such that each entry is a tuple in a  $\mathcal P_H$. This leads to the following isomorphism of $G_p$-sets
$$\mathcal{P}^{\leq\beta}\cong\bigsqcup_{\{i_H\geq 0 \mid H\in \SUB_{G_p}(I^t), \sum i_H\leq \beta\}}\left(\prod_{H\in\SUB_{G_p}(I^t)}\mathcal{P}_H^{i_H}\right).$$
%Then $$\mathcal{P}_j^{\leq\beta}/G_j\cong\bigsqcup_{\{(i_\xi)\in (\Z/p_j)^{\mathfrak{P}(\{p_1,\ldots,p_r\}-\{p_j\})}\mid \sum i_\xi\leq \beta\}}\left(\prod_{\xi\subseteq\{p_1,\ldots,p_r\}-\{p_j\}}\mathcal{P}_\xi^{i_\xi}\right)/G_j.$$

The set $$\prod_{H\in\SUB_{G_p}(I^t)}\mathcal{P}_H^{i_H}$$ can be expressed as a disjoint union of $G_p$-orbits in a similar way as in the proof of \Cref{lemma: P_j:orbits}. For the sake of completeness we include the whole argument. Anyway notice it is completely analogous to the proof of \Cref{lemma: P_j:orbits}.

For each subgroup $H$ of $G$ we have the following decomposition of $G_p$-sets $$\mathcal{P}^{i_H}_H=(\mathcal{P}^{i_H}_H-\{0\}_H)\sqcup\{0\}_H,$$ where $\{0\}_H$ denotes the $G_p$-conjugacy class of $\{0\}\rtimes\Z/p$ in $N_H\rtimes\Z/p$, in particular $\{0\}_H$ is a singleton.

Denote by $\mathcal{G}_p$ the collection of all subgroups of $G_p$. By \Cref{lemma:P_j:product} and the distributivity property of cartesian product over disjoint union we have the following isomorphisms of $G_p$-sets

\begin{align*}
 \prod_{H\leq G_p}\mathcal{P}^{i_H}_H
    &\cong \bigsqcup_{ \mathcal{I}\subseteq   \mathcal{G}_p}\left( \prod_{H\in \mathcal{I}}  (\mathcal{P}^{i_H}_H-\{0\}_H)\times \prod_{K\in\mathcal{G}_p- \mathcal{I}} \{0\}_K \right),\\
    &\cong \bigsqcup_{ \mathcal{I}\subseteq   \mathcal{G}_p} \prod_{H\in \mathcal{I}}  (\mathcal{P}^{i_H}_H-\{0\}_H)\\
    &\cong \bigsqcup_{ A\subseteq   \SUB_{G_p}(I^t) } \prod_{H\in A}  (\mathcal{P}^{i_H}_H-\{0\}_H)
\end{align*}

In the proof of \Cref{lemma: P_j:orbits} it is shown that every element in  $\mathcal{P}_H-\{0\}_H$ has isotropy $H$, we conclude that $\mathcal{P}^{i_H}_H-\{0\}_H$ has constant isotropy $H$ and $\prod_{H\in A}  (\mathcal{P}^{i_H}_H-\{0\}_H)$ has constant isotropy $\cap A$. Thus we have \begin{align*}\left|\left(\prod_{H\leq G_p}\mathcal{P}_H^{i_H}\right)/G_p\right|&=\sum_{A\subseteq \SUB_{G_p}(I^t)}\prod_{H\in A}\dfrac{(\left|\mathcal{P}_H^{i_H}\right|-1)|\cap A|}{|G_p|}.\end{align*}
Now the claim follows.
\end{proof}

\section{The $p$-torsion of $H^l(\Z^n\rtimes G)$}\label{section:ptorsion}

Now we have a complete computation of the $p$-torsion part of $H^l(\Z^n\rtimes G)$. Given a $\Z[G]$-module $M$, we say that it is \emph{virtually of type $(r,s,t)$ with respect to $p$} if $M$, thought as a $\Z[\Z/p]$, admits  a finite index $\Z[\Z/p]$-module $N$ of type $(r,s,t)$. Note that \cref{prop:virtual:rst:decompositions} guarantees that every finitely generated $\Z[G]$-module is virtually of type $(r,s,t)$ with respect to $p$.

\begin{theorem}\label{main-thm} Suppose that $\Z^n$, thought of as a $\Z[G]$-module, virtually of type $(r,s,t)$ with respect to $p$.
For any $l$ the $p$-torsion part of $H^l(\Z^n\rtimes G)$ is isomorphic to
\begin{itemize}
\item For $l$ even: $(\Z/p)^\theta$,  where $$\theta=\sum_{\substack{l_1+l_2=l\\ l_2\text{ even}}}\sum_{A\subseteq\SUB_{G_p}(I^t)}\left(\sum_{\tau=0}^s\theta(A,l_2-p\tau)\binom{s}{\tau}\right)\cdot\rank_{\Z}(H^{l_1}(\Z^r)^{\cap A})$$
    \item For $l$ odd:
    $(\Z/p)^\theta$,  where $$\theta=\sum_{\substack{l_1+l_2=l\\ l_2\text{ odd}}}\sum_{A\subseteq\SUB_{G_p}(I^t)}\left(\sum_{\tau=0}^s\theta(A,l_2-p\tau)\binom{s}{\tau}\right)\cdot\rank_{\Z}(H^{l_1}(\Z^r)^{\cap A})$$
\end{itemize}
\end{theorem}
\begin{proof} 
Let us denote the $p$-torsion part of $H^l(\Z^n\rtimes G)$ by $T_p^l$. Since tensoring with $\Z_{(p)}$ does not modify the $p$-torsion, we have that the $p$-torsion subgroup of $H^l(\Z^n\rtimes G)_{(p)}$ is isomorphic to $T^l_p$.  By \Cref{lemma:reducing:to:invariants} we get,
for all  $l\geq 0$, that
 \begin{equation*}
H^l(\Z^n\rtimes G)_{(p)}\cong \left(H^l(\Z^n\rtimes\Z/p)_{(p)}\right)^{G_p}. 
 \end{equation*} 
 By \cref{lemma:rst}, we may assume $\Z^n$ is of type $(r,s,t)$, that is, $\Z^n$ is isomorphic as a $\Z[G]$-module to $\Z^r\oplus \Z[\Z/p]^s\oplus I^t$. Since the $\Z/p$ action on $\Z^r$ is trivial we have that $\Z^n\rtimes \Z/p$ is isomorphic to $\Z^r\times ((\Z[\Z/p]^s\oplus I^t)\rtimes \Z/p)$. Thus applying the Künneth formula we have 
 $$H^l(\Z^r\times ((\Z[\Z/p]^s\oplus I^t)\rtimes \Z/p))_{(p)}\cong\bigoplus_{\alpha+\delta=l} H^\alpha(\Z^r)\otimes_{\Z} H^\delta((\Z[\Z/p]^s\oplus I^t)\rtimes \Z/p)_{(p)}$$
Notice that the $G_p$ conjugation action on $\Z^r\times ((\Z[\Z/p]^s\oplus I^t)\rtimes \Z/p)$ respects the product decomposition since the $(r,s,t)$ decomposition of $\Z^n$ is $G$-equivariant. Hence the last isomorphism is a splitting of $\Z[G_p]$-modules. Moreover the $G_p$ action on $H^\alpha(\Z^r)\otimes_{\Z} H^\delta((\Z[\Z/p]^s\oplus I^t)\rtimes \Z/p)_{(p)}$ is diagonal.
 
 Combining these last two isomorphisms  we obtain
 \begin{align*}
H^l(\Z^n\rtimes G)_{(p)}&\cong\left( \bigoplus_{\alpha+\delta=l} H^\alpha(\Z^r)\otimes_{\Z} H^\delta((\Z[\Z/p]^s\oplus I^t)\rtimes \Z/p)_{(p)}\right)^{G_p}\\&\cong \bigoplus_{\alpha+\delta=l} (H^\alpha(\Z^r)\otimes_{\Z} H^\delta((\Z[\Z/p]^s\oplus I^t)\rtimes \Z/p)_{(p)})^{G_p}
\end{align*}

Since $H^*(\Z^r)$ is a torsion free abelian group, \cref{cor:torsion:st} gives us the following isomorphism
\[T_p^l\cong \bigoplus_{\alpha+\delta=l}\left(H^\alpha(\Z^r)\otimes\left(\bigoplus_{\tau=0}^s(\mathcal{A}_p^{\delta-p\tau})^{\binom{s}{\tau}}\right)\right)^{G_p}\]
where $\mathcal{A}^{\delta-p\tau}_p$ is the torsion subgroup of $H^{\delta-p\tau}(I^t\rtimes\Z/p)$.
By Corollary \ref{permutation-module} we have a graded isomorphism of $\Z[G_p]$-modules
    \begin{align*}
        \mathcal{A}_p^{\delta-p\tau}&\cong \bigoplus_{\emptyset\neq A\subseteq\SUB_{G_p}(I^t)}(\Z/p[G_p/\cap A])^{\theta(A,\delta-p\tau)}
    \end{align*}
    whenever $\delta-p\tau$ is even and 0 for $\delta-p\tau$ odd.
    Then \begin{align*}
        T_p^{l}&\cong\bigoplus_{\alpha+\delta=l}\left(\bigoplus_{\emptyset\neq A\subseteq\SUB_{G_p}(I^t)} H^\alpha(\Z^r)\otimes \left(\bigoplus_{\tau=0}^s\Z/p[G_p/\cap A]^{\theta(A,\delta+p\tau)  }\right)\right)^{G_p}\\
    &\cong\bigoplus_{\alpha+\delta=l}\left(\bigoplus_{\emptyset\neq A\subseteq\SUB_{G_p}(I^t)}\bigoplus_{\tau=0}^s H^\alpha(\Z^r)\otimes \Z/p[G_p/\cap A]^{\theta(A,\delta-p\tau)  }\right)^{G_p}\\
    &\cong\bigoplus_{\alpha+\delta=l}\left(\bigoplus_{\emptyset\neq A\subseteq\SUB_{G_p}(I^t)}\bigoplus_{\tau=0}^s H^\alpha(\Z^r)\otimes \Z[G_p/\cap A]^{\theta(A,\delta-p\tau)  }\otimes \Z/p\right)^{G_p}\\
    &\cong\bigoplus_{\alpha+\delta=l}\left(\bigoplus_{\emptyset\neq A\subseteq\SUB_{G_p}(I^t)}\bigoplus_{\tau=0}^s (H^\alpha(\Z^r)\otimes \Z[G_p/\cap A])^{\theta(A,\delta-p\tau)  }\right)^{G_p}\otimes \Z/p
    \end{align*}

    By Proposition 6.2 on page 73 in \cite{BrownBook}, we have $$H^\alpha(\Z^r)\otimes \Z[G_p/\cap A]\cong \operatorname{Ind}_{\cap A}^{G_p}\operatorname{Res}_{\cap A}^{G_p}H^\alpha(\Z^r).$$ Thus we have

    \begin{align*}
        T^l_p &\cong \bigoplus_{\alpha+\delta=l}\left(\bigoplus_{\emptyset\neq A\subseteq\SUB_{G_p}(I^t)}\bigoplus_{\tau=0}^s (\operatorname{Ind}_{\cap A}^{G_p}\operatorname{Res}_{\cap A}^{G_p}H^\alpha(\Z^r))^{\theta(A,\delta-p\tau)  }\right)^{G_p}\otimes \Z/p\\
        &\cong \bigoplus_{\alpha+\delta=l}\left(\bigoplus_{\emptyset\neq A\subseteq\SUB_{G_p}(I^t)} (\operatorname{Ind}_{\cap A}^{G_p}\operatorname{Res}_{\cap A}^{G_p}H^\alpha(\Z^r))^{\sum_{\tau=0}^s\theta(A,\delta-p\tau)  }\right)^{G_p}\otimes \Z/p\\
        &\cong \bigoplus_{\alpha+\delta=l}\quad\bigoplus_{\emptyset\neq A\subseteq\SUB_{G_p}(I^t)} \left((\operatorname{Ind}_{\cap A}^{G_p}\operatorname{Res}_{\cap A}^{G_p}H^\alpha(\Z^r))^{\sum_{\tau=0}^s\theta(A,\delta-p\tau)  }\right)^{G_p}\otimes \Z/p\\
        &\cong \bigoplus_{\alpha+\delta=l}\quad\bigoplus_{\emptyset\neq A\subseteq\SUB_{G_p}(I^t)} \left((\operatorname{Ind}_{\cap A}^{G_p}\operatorname{Res}_{\cap A}^{G_p}H^\alpha(\Z^r))^{G_p  }\right)^{\sum_{\tau=0}^s\theta(A,\delta-p\tau)}\otimes \Z/p
    \end{align*}
    
    Notice that $(\operatorname{Ind}_{\cap A}^{G_p}\operatorname{Res}_{\cap A}^{G_p}H^\alpha(\Z^r))^{G_p  }$ is isomorphic to $H^0(G_p; \operatorname{Ind}_{\cap A}^{G_p}\operatorname{Res}_{\cap A}^{G_p}H^\alpha(\Z^r))$ and this group is isomorphic, by Shapiro's lemma, to $H^0(\cap A; \operatorname{Res}_{\cap A}^{G_p}H^\alpha(\Z^r))\cong H^\alpha(\Z^r)^{\cap A}$. Thus 

    \begin{align*}
        T^l_p &\cong \bigoplus_{\alpha+\delta=l}\bigoplus_{A\subseteq\SUB(G_p))} \left( H^\alpha(\Z^r)^{\cap A} \right)^{\sum_{\tau=0}^s\theta(A,\delta-p\tau)}\otimes \Z/p
    \end{align*}
\end{proof}

\section{How to go from \cref{main-thm} to \cref{main-thm:intro}}\label{section:from:main:to:intro}

In this section we explain how to translate the notation inside the statement of \cref{main-thm} to obtain \cref{main-thm:intro} from the introduction of this paper.

\subsection{From $\SUB_{G_p}(I^t)$ to $\mathcal D$}

Recall that $\SUB_{G_p}(I^t)$ is the collection of all subgroups of $G_p$ that appear as isotropy groups of nonzero elements of $I^t$. 

The collection of all subgroups of the cyclic group $G_p$ is in bijection with the set of all positive divisors of $m/p$; this bijection is given by sending a subgroup to its order.

Fix a generator $g$ of $G_p$ and consider the complexification $\phi_t=g\otimes \mathrm{Id}:I^t\otimes \mathbb C\to I^t\otimes \mathbb C$. This complexification determines an action of $G_p$ on $I^t\otimes \mathbb C$, and note that the set of isotropy groups of nonzero elements of this action coincides with $\SUB_{G_p}(I^t)$. Let $H$ be the subgroup of $G_p$ of order $d$. Hence $H$ is generated by $g^{m/pd}$. Assume $g^{m/pd}$ is represented by a diagonal matrix $\mathbf A$. Note that $H$ fixes a nonzero element of $I^t\otimes \mathbb C$ if and only if  $\mathbf A^{m/pd}$ admits 1 as an eigenvalue. Hence, $H$ appears as the isotropy group of a nonzero element of $I^t\otimes \mathbb C$ if and only if $\mathbf A$ admits a primitive $m/pd$-root of unity as an eigenvalue.

As a conclusion, there is a bijection between $\SUB_{G_p}(I^t)$ and the set of positive divisors $d$ of $m/p$ such that there is a primitive $\frac{m}{pd}$-th root of unity among the complex eigenvalues $\phi_t$.

\subsection{From $k_H$ to $\mathsf k_d$} For this subsection we keep the notation from the previous subsection. Let us recall the definition of $k_H$ for a subgroup $H$ of $G_p$. By \Cref{cor:decomposition:constant:isotropy:v3} we obtain a finite index submodule $\bigoplus_{H\leq G_p}N_H$ of $I^t$ where every non-zero element of $I^t_H$ has  isotropy $H$. By definition $k_H=\rank(N_H)/(p-1)$.

After tensoring with $\mathbb C$ we get $I^t\otimes \mathbb C=\bigoplus_{H\leq G_p}N_H\otimes \mathbb C$, and $\rank(N_H)=\dim_{\mathbb C}(N_H\otimes \mathbb C)$. Note that the splitting of $I^t\otimes \mathbb C$ yields a diagonalization of $\phi_t$ by block (diagonal matrices) $\mathbf A=\bigoplus_{H\leq G_p} \mathbf A_H$, where the elements in the diagonal of $\mathbf A_H$ are primitive $m/pd$-th roots of unity.
Therefore $\dim_{\mathbb C}(N_H\otimes \mathbb C)$ is the number $m_d$ of eigenvalues of $\phi_t$ that are $\frac{m}{pd}$-th roots of unity, and $k_H=k_d:=m_d/(p-1)$. 

\subsection{From $\mathcal P_H^{i}$ to $\mathsf P_d^i$} Let $H\leq G_p$ be a subgroup and $\beta$ be a nonnegative integer. Recall that $k_H = \rank(N_H)/(p-1)$, see \eqref{conteo:orbitas}. Also recall that there is a bijection between $\mathcal{P}_H$ and $(\Z/p)^{k_H}$. Consider the following sets

$$\mathcal{P}_H^i=\{([x_1],\ldots,[x_{k_H}])\in(\Z/p)^{k_H}\mid 0\leq x_j \leq p-1 \text{ and }  \sum_jx_j=i\}.$$

As in the introduction, define  $$\mathsf P_d^i=\sum_{j=0}^{\mathsf k_d} (-1)^j\binom{\mathsf k_d}{j}\binom{i-jp+\mathsf k_d-1}{\mathsf k_d-1}. $$ 
Here we use the convention that $\binom{a}{b}=0$ whenever $a<0$ or $b<0$.

\begin{lemma}
    Let $H$ be a subgroup of $G_p$ of order $d$. Then $\mathsf P^i_d=|\mathcal P_H^i|$.
\end{lemma}
\begin{proof}
    This is a classical problem in combinatorics. For the sake of readability, we include a sketch of the argument.

    First note that the cardinality of the set $\mathcal X$ of $\mathsf k_d$-tuples of nonnegative integers whose sum is $i$ is given by
    \[\binom{i+\mathsf k_d-1}{\mathsf k_d-1}.\]
    Next we want to impose the restriction given by the inequalities $x_j\leq p-1$ for each $0\leq j\leq \mathsf k_d$. It is enough to compute the cardinality of the set $\mathcal Y=\mathcal X - \mathcal P_H^i$ of sums with at least one of the $x_j$ greater than or equal to $p$. We split $\mathcal Y$ as a union $\mathcal Y_1 \cup \cdots \cup \mathcal Y_{\mathsf k_d}$ where $\mathcal{Y}_t$ is the set of all tuples in $\mathcal Y$ such that $x_t\geq p$. The cardinality of $\mathcal Y_{t_1}\cap \cdots \cap \mathcal Y_{t_j}$ is 
    \[\binom{i-jp+\mathsf k_d -1}{\mathsf k_d -1}.\]
    Now the result follows by applying the inclusion-exclusion principle to the splitting of $\mathcal Y$.
\end{proof}

\subsection{From $\theta(A,\beta)$ to $\mathsf T(A,\beta)$} The fact that $\theta(A,\beta)=\mathsf T(A,\beta)$ follows from the previous subsections.

\subsection{From $\rank_{\Z}(H^{l}(\Z^r)^{\cap A})$ to $\mathsf H(l,d,\Z^r)$} 

\begin{lemma}\label{lem:rank:Hld}
    Let $M$ be a $\Z[G]$-lattice, and let $H$ be the subgroup of $G$ of order $d$. Then $\rank_{\Z}(H^{l}(M)^{H})=\mathsf H(l,d,M)$.
\end{lemma}
\begin{proof} Denote $r(M)=\rank_{\Z}(M)$, that is $M=\Z^{r(M)}$ as an abelian group.
     As before we work with the complexification $H^{l}(M)\otimes \mathbb C$ which is isomorphic to the $l$-th exterior power $\wedge^l \mathbb C^{r(M)}$. We endow $\wedge^l \mathbb C^{r(M)}$ with the natural induced $G$-action. Let us compute $\dim_{\mathbb C}((\wedge^l \mathbb C^{r(M)})^H)$ which is equal to $\rank_{\Z}(H^{l}(M)^{H})$, since tensoring with $\mathbb C$ and taking $H$-invariants are operations that commute as the $G$-action on $\mathbb C$ is trivial.

Let $g$ a generator of $G$. Let $v_1,...,v_{r(M)}$ be a basis of eigenvectors with corresponding eigenvalues $\lambda_1,..., \lambda_{r(M)}$ for $g$, and consider the associated basis for $\wedge^l \mathbb C^{r(M)}$ consisting of the products $v_{i_1}\wedge \cdots \wedge v_{i_l}$ for $1\leq i_1 < \cdots < i_l\leq r(M)$. This basis for $\wedge^l \mathbb C^{r(M)}$ is also a basis of eigenvectors, and the eigenvalue associated to  $v_{i_1}\wedge \cdots \wedge v_{i_l}$ is $\lambda_{i_1}\cdots \lambda_{i_l}$. Hence $(\wedge^l \mathbb C^r)^H$ is spanned by the basis elements of  $\wedge^l \mathbb C^{r(M)}$ whose associated eigenvalues are $m/d$-roots of unity. Therefore the dimension of 
$(\wedge^l \mathbb C^{r(M)})^H$ equals the number $\mathsf H(l,d, M)$ defined in the introduction.
\end{proof}

\section{Computation of the rank}\label{section:rank}

This section is devoted to the computation of the rank of $H^l(\Z^n\rtimes G)$ for each $l\geq 0$.\\

\begin{theorem}\label{rank}
    For each $l\geq 0$, we have
    \[\rank_{\Z}(H^l(\Z^n\rtimes G))= \mathsf H(l, m, \Z^n). \]
\end{theorem}
\begin{proof}
A straightforward application of the Lyndon-Hochschild-Serre spectral sequence for the canonical extension associated to $\Z^n\rtimes G$ (compare with the proof of \cref{lemma:reducing:to:invariants}) yields the following isomorphism
\[\rank_{\Z} (H^l(\Z^n\rtimes G))=\rank_{\Z} (H^l(\Z^n)^G).\]

Now the result follows from \cref{lem:rank:Hld}.
\end{proof}

\section{How to go from the general setting to the $(r,s,t)$ case}\label{section:generaltorst}

\subsection{Replacing $\Z^n$ with a finite index submodule of type $(r,s,t)$}

\begin{lemma}[Replacing by a module of type $(r,s,t)$]\label{lemma:rst}
     Given a $\Z[G]$-lattice $M$, there is a $\Z[G]$-module $M'$ and an injective $\Z[G]$-homomorphism  $f:M'\to M$ such that
    
    \begin{enumerate}
        \item $M'$ is of type $(r,s,t)$ and the corresponding decomposition is $G$-invariant,
        \item $f$ is an isomorphism after tensoring with $\Z_{(p)}$,
        \item $f$ induces isomorphisms of $\Z_{(p)}[G]$-modules
    $$H^*(M'\rtimes\Z/p)_{(p)}\cong H^*(M\rtimes\Z/p)_{(p)}$$
    \end{enumerate}
\end{lemma}
\begin{proof}
The existence of $f\colon M'\to M$ satisfying the first and  second item is a direct consequence of \cref{prop:virtual:rst:decompositions}. The property in the third item is an immediate consequence of comparing the Leray-Serre spectral sequences associated to the extensions
$$BM'\to B(M'\rtimes \Z/p)\to B\Z/p$$
and
$$BM\to B(M\rtimes \Z/p)\to B\Z/p.$$
\end{proof}

By \Cref{lemma:rst}, the computation of $\left(H^l(\Z^n\rtimes\Z/p)_{(p)}\right)^{G_p}$ can be done after replacing the $M=\Z^n$ (thought of as a $\Z[\Z/p]$-module)  by a $\Z[\Z/p]$-module of type $(r,s,t)$ with $G$-invariant decomposition.

\subsection{How to compute $r$, $t$, $\phi_r$ and $\phi_t$}
Assume that the $G$-module $\Z^n$ splits $G$-equivariantly as $\Z^r\oplus \Z[\Z/p]\oplus I^t$.
Fix a generator $g$ of $G$ and let $\phi$ be the matrix with integer coefficients, associated with $g\colon \Z^n\to \Z^n$ with respect to the canonical basis. 

Note that we have the following isomorphisms
    \begin{align*}H^i(\Z/p;\Z^n) 
        &\cong H^i(\Z/p;\Z^r\oplus \Z[\Z/p]^s\oplus I^t)\\
        &\cong H^i(\Z/p;\Z^r)\oplus H^i(\Z/p;\ \Z[\Z/p]^s)\oplus H^i(\Z/p;I^t)\\
        &\cong H^i(\Z/p;\Z^r)\oplus H^i(\Z/p;I^t)
    \end{align*}
where $H^i(\Z/p;\Z[\Z/p])$ vanishes since $\Z[\Z/p]$ is flat. Notice that $H^1(\Z/p;\Z^r)$ vanishes since the coefficients have a trivial $\Z/p$-action, and $H^2(\Z/p,I^t)$ vanishes as a straightforward computation using the long exact sequence associated with the short exact sequence of coefficients $0\to I\to \Z[\Z/p] \to \Z\to 0$. Also, a direct computation yields $H^1(\Z/p;I)\cong \Z/p$ and $H^2(\Z/p;\Z)\cong \Z/p$. As a conclusion

\begin{equation}\label{eq:cohomology}
    H^i(\Z/p;\Z^n)\cong
\begin{cases}
    (\Z/p)^r & i=1 \\
    (\Z/p)^t & i=2
\end{cases}
\end{equation}

Let us explain how to compute $r$ and $\phi_r$.  Using the classical procedure to compute the cohomology groups of a cyclic group, we have that
    \[H^1(\Z/p;\Z^n)\cong \ker(N)/\mathrm{Im}(\phi-I)\]
    where $I$ is the $n\times n$ identity matrix and $N=I+\phi+\cdots +\phi^{p-1}$.

    By a well-known result in finitely generated $\Z$-modules, there is a basis \begin{equation}\label{eq:basis}
        \{v_1,\dots , v_k,w_1,\dots ,w_t\}
    \end{equation} of $\ker(N)$ such that $\{v_1,\dots , v_k,pw_1,\dots ,pw_t\}$ is a basis for $\mathrm{Im}(\phi-I)$. Using Lemma 1.9 (iii) in \cite{DL13}, we conclude that $\{w_1,\dots ,w_t\}$ generates $I^t$ as a $\Z[\Z/p]$-module. Let $\mathbf B$ be the matrix that results from $\phi$ after changing the basis from the canonical one on $\Z^n$ to the basis \eqref{eq:basis}. Since both, the subgroups generated by $\{v_1,\dots , v_k\}$ and $\{w_1,\dots ,w_t\}$ are $G$-invariant, we have that $\mathbf B$ is a two-blocked matrix. The second block gives the desired isomorphism $\phi_t:I^t\to I^t$.

    In order to compute $r$ and $\phi_r$ we run a completely analogous procedure but now using the isomorphism
        \[H^2(\Z/p;\Z^n)\cong \ker(\phi-I)/\mathrm{Im}(N).\]

\begin{remark}\label{remark:uniqueness:rst}
    From equation \eqref{eq:cohomology} we conclude that the constants $r$ and $t$ are completely determined by the $\Z/p$-module structure of $\Z^n$. Since $r+ps+(p-1)t=n$ we also conclude that $s$ is also determined by $\Z^n$.
\end{remark}

\section{The one prime case}\label{un primo}

Assume $G\cong \Z/p$ with $p$ prime. Hence $G_p$ is the trivial group. In this case $H^l(\Z^n\rtimes \Z/p)$ has the form $\Z^w\oplus (\Z/p)^\theta$. First, let us apply \cref{main-thm:intro} to compute $\theta$. All of the following come from the definitions applied to this concrete case. Let us denote by $\mathbf{1}$ the trivial group.

\begin{itemize}
    \item $\mathcal{D}=\{1\}$.
    \item $\mathsf k_{1}=\rank_\Z(I^t)/(p-1)=t$.
    \item For $i\geq 0$, we have 
    $$\mathsf P_1^i=\sum_{j=0}^{t} (-1)^j\binom{t}{j}\binom{i-jp+t-1}{t-1}. $$
    \item The only possibilities for $A\subseteq \mathcal D$ are $A=\emptyset$, and $A=\{1\}$.
    \item Let $\beta\geq 0$. When $A=\emptyset$ we have $\mathsf T(A,\beta)=1$. When $A=\{1\}$ we get
    \[ \mathsf T(A,\beta)=\sum_{0\leq i\leq \beta} (\mathsf{P}_1^i-1)= \sum_{0\leq i\leq \beta} \mathsf{P}_1^i-\beta-1.\]
    \item $\mathsf H(l,1,\Z^r)=\binom{r}{l}$.
\end{itemize}

Thus, plunging in this data to \cref{main-thm:intro} we get for $l$ even (resp. odd)
\begin{align*}
    \theta &=\sum_{\substack{l_1+l_2=l\\ l_2\text{ even (resp. odd)}}}\left( \sum_{\tau=0}^s \mathsf T(\{\mathbf{1}\},  l_2-p\tau) \binom{s}{\tau} \right)\rank_\Z(H^{l_1}(\Z^r)^{\mathbf{1}})\\
    &= \sum_{\substack{l_1+l_2=l\\ l_2\text{ even (resp. odd)}}}\left( \sum_{\tau=0}^s \left( 1+\sum_{0\leq i\leq l_2-p\tau} \mathsf{P}_1^i-l_2+p\tau \right) \binom{s}{\tau} \right)\binom{r}{l_1}
\end{align*}

In order to compute the rank of $H^l(\Z^n\rtimes \Z/p)$ we need to know the complex eigenvalues for a generator $g$ of $G$. Hence such a rank equals $\mathsf H (l,p,\Z^n)$, which is exactly the multiplicity of the eigenvalue 1 in the $l$-th exterior power of $g$.

\section{An explicit example}\label{section:explicit:example}
Let us consider the semidirect product $\Z^5\rtimes \Z/6$ where the action of $\Z/6$ on $\Z^5$ where the action of a  generator of $\Z/6$ is given by the matrix

\[D=
\begin{pmatrix}
-1 & 0 & 0&0&0\\
0&0&1 & 0 & 0\\
0 &1& 0 & 0&0\\0&0&0&0&-1\\
0&0&0&1&-1
\end{pmatrix}
\]
To compute the torsion of the cohomology groups of $\Z^5\rtimes \Z/6$ we use  \cref{main-thm}.

\subsection{Computation of the 2-torsion}Let $p=2$. First we need to find the $(r,s,t)$ decomposition for the $\Z/2$-action  on $\Z^5$ which is given by the following matrix 
\[D^3=
\begin{pmatrix}
-1 & 0 & 0&0&0\\
0&0&1 & 0 & 0\\
0 &1& 0 & 0&0\\0&0&0&1&0\\
0&0&0&0&1
\end{pmatrix}
\]
 Then $r=2$, $s=1$ and $t=1$ and $I^t=\langle (1,0,0,0,0)\rangle$. Note that $G_2=\Z/3$ is acting trivially on $I^t$, thus $\mathsf k_{3}=1$ and $\mathsf k_{ 1}=0$, we conclude that $\mathcal D=\{3\}$. The numbers $\mathsf T(A,\beta)$ only have two possibilities for $\beta$ even:

\begin{itemize}
\item $\mathsf T(\emptyset, \beta)=1$  when  $A=\emptyset$, and
\item $\mathsf T(\{3\},\beta)=0$, since $\mathcal{P}_{\Z/3}^{i_{\Z/3}}$ is a singleton.
\end{itemize}
We conclude that the 2-torsion of $H^l(\Z^5\rtimes\Z/6)$ is isomorphic to $(\Z/2)^{\theta(2)}$ where
\begin{align*}
    \theta(2)=\begin{cases}\sum_{\substack{l_1+l_2=l\\ l_2\text{ even}}}\left(\mathsf T(\emptyset,l_2)+\mathsf T(\emptyset,l_2-2)\right)\mathsf H(l_1,3,\Z^r)& \text{ for $l$ even}\\0& \text{ for $l$ odd }.\end{cases}
\end{align*}
Note that $\Z/3$ acts freely outside the origin on $\Z^r$, then by \cite[Lemma 1.22 (iii)]{DL13}
$$\mathsf H(l_1,3,\Z^r)=\rank_\Z(H^{l_1}(\Z^r)^{\Z/3})=\begin{cases}
    1&\text{ if }l_1=0,2\\0 & \text{ if }l_1\neq 0,2
\end{cases}$$
Then we get 
$$\theta(2)=\begin{cases}
    2&\text{ if $l$ is even and } l>0\\0 & \text{ otherwise.}
\end{cases}$$

\subsection{Computation of the 3-torsion}
Now suppose $p=3$. The $\Z/3$-action on $\Z^5$ is given by the matrix
\[D^2=
\begin{pmatrix}
1 & 0 & 0&0&0\\
0&1&0 & 0 & 0\\
0 &0& 1 & 0&0\\0&0&0&-1&1\\
0&0&0&-1&0
\end{pmatrix}
\]
In this case $r=3$, $s=0$ and $t=1$. Moreover $I^t=\langle(0,0,0,1,0),(0,0,0,0,1)\rangle$ and $G_3=\Z/2$ act trivially on $I^t$, then $\mathsf k_{1}=0$ and $\mathsf k_{2}=1$.
Hence $\mathcal D=\{\Z/2\}$, and the numbers $\mathsf T(A,\beta)$ only have two possibilities for $\beta$ even:
\begin{itemize}
 \item  $\mathsf T(\emptyset, \beta)=1$ when $A=\emptyset$, and
\item $\mathsf T(\{\Z/2\},\beta)=0$, because  $\mathcal{P}_{\Z/2}^{i_{\Z/2}}$ is a singleton.
\end{itemize}
We conclude that the 3-torsion of $H^l(\Z^5\rtimes\Z/6)$ is given by $(\Z/3)^{\theta(3)}$, where 

$$\theta(3)=\begin{cases}\sum_{\substack{l_1+l_2=l\\ l_2\text{ even}}} \mathsf H (l_1,2,\Z^r)& \text{ if $l$ is even}\\0&\text{ if $l$ is odd}\end{cases}$$

Note that we have a decomposition of $\Z/2$-modules
$$\Z^r=\langle(1,0,0,0,0)\rangle\oplus\langle(0,1,0,0,0),(0,0,1,0,0)\rangle $$ where $\Z/2$ acts free outside the origin on $\langle(1,0,0,0,0)\rangle$ and $\langle(0,1,0,0,0),(0,0,1,0,0)\rangle \cong \Z[\Z/2]$ as $\Z/2$-modules. Then the Künneth formula gives us 
$$H^{l_1}(\Z^r)^{\Z/2}=\begin{cases}
    \Z&\text{ if $l_1=0$}\\
    %\Z&\text{ if } l_1=1\\
    0&\text { if } l_1>0
\end{cases}$$Then we have that 
$$\theta(3)=\begin{cases}
    1&\text{ if $l$ is even and } l>0\\0 & \text{ in any other case.}
\end{cases}$$

\subsection{Computation of the rank} A hands-on application of \cref{rank} leads to:

$$\rank_\Z(H^l(\Z^5\rtimes\Z/6))=\begin{cases}
    1&l=0,1,4,5\\2&l=2,3\\0&l>5
\end{cases}$$
Bringing all together, we get the following.
\begin{theorem}
    The group cohomology of the group $\Z^5\rtimes\Z/6$, where the action is defined by the matrix $D$, is given by
    $$H^l(\Z^5\rtimes\Z/6)\cong\begin{cases}
        \Z&l=0\\
        \Z&l=1\\
        \Z^2\oplus\Z/2^2\oplus\Z/3&l=2\\
        \Z^2&l=3\\
        \Z\oplus\Z/2^2\oplus\Z/3&l=4\\
        \Z&l=5\\
        \Z/2^2\oplus\Z/3&l>5, l\text{ even}\\
        0&l>5, l\text{ odd}
    \end{cases}$$
\end{theorem}

\appendix

\section{Some splitting results for $\Z[G]$-lattices}\label{Appendix}

The main goal of this appendix is to prove some splitting results for $\Z[G]$-lattices that might be well known for the experts on representation theory of cyclic groups, but we were not able to find them in the literature.

\subsection{Factorization in modules of constant isotropy}

The proof of the following theorem follows from the classical proof of Maschke's theorem, see for instance \cite[Theorem 1]{Serre}.

\begin{lemma}[Maschke's Theorem]\label{virtual:maschke}
    Let $H$ be a finite group and let $R$ be a ring. Consider $M$ an $R[H]$-module and $V$ an $R[H]$-submodule such that it is a direct summand of $M$ considered as an $R$-module. Then there exists an $R[H]$-submodule $W$ of $M$ such that the following equality of $R[H]$-modules holds $|H|\cdot M=(|H|\cdot V)\oplus W$.
\end{lemma}

\begin{definition}
Let $M$ be a torsion free $R[H]$-module. We say \emph{$M$ has constant isotropy} if any two nonzero elements in $H$ have the same $H$-isotropy group. Such an isotropy group will be called \emph{the isotropy of $M$}.
\end{definition}

For the next theorem let $G$ and $G_p$ be as in \cref{notation}. Denote by $\Z_{(p)}$ the subring of $\mathbb Q$ of all fractions with denominator prime relative to $p$. 

\begin{theorem}\label{cor:decomposition:constant:isotropy:v3} Let $M$ be a $\Z[G]$-lattice. Suppose the $\Z/p$-action on $M$ is free outside the origin.  Then for each $H\leq G$ there is a (possibly trivial)  $\Z[G]$-submodule $M_H$ of $M$ with constant isotropy $H$ such that $\bigoplus_{H\leq G_p} M_{H}$ has finite index in $M$ which is coprime with $p$. In particular, we have an isomorphism $M\otimes \Z_{(p)} \cong \bigoplus_{H\leq G} M_{H} \otimes \Z_{(p)}$ of $\Z_{(p)}[G]$-modules.

%Moreover, the decomposition is unique, i.e. each $M_H$ is completely determined by $M$ and $H$.
\end{theorem}
\begin{proof} 
We proceed by induction on the $\Z$-rank of $M$. 
If $\rank(M)=1$, then for $x,y\in M-\{0\}$, there is $r\in \Z-\{0\}$ such that 
$y=rx$. We claim that $G_y=G_x$. Since the $G$-action on $M$ is $\Z$-linear, we have $G_x\subseteq G_y$. For $g\in G_y$, we have $gy-y=r(gx-x)=0$, and as $\Z$ is an integer domain we get $gx=x$. Therefore $G_y\subseteq G_x$.

Now suppose $M$ is a $\Z[G]$-lattice with $\Z$-rank at least 2. Let us consider the following family of subgroups of $G$ ordered by inclusion
$$\mathcal{F}=\{H\leq G\mid M^H\neq 0\},$$
where $M^H$ are the $H$-invariants of $M$.
Let $H_{max}$ be a maximal element in $\mathcal{F}$. Notice that $M^{H_{max}}$ is an $\Z[G]$-module with constant isotropy $H_{max}$.  Note that if  $rm\in M^{H_{max}}$ with $ m\in M$ and $ r\in \Z-\{0\}$, then $m\in M^{H_{max}}$. The latter implies that $M/M^{H_{max}}$ is torsion free. Notice that for an element $m\in M$ and norm element $N$ of $\Z[\Z/p]$, we have that $Nm$ is a $\Z/p$-fixed point of $M$. As the $\Z/p$-action on $M$ is free outside of the origin, we have that the norm element $Nm=0$; as a conclusion we have that $M$ is a $\Z[\Z/p]/\langle N \rangle\cong \Z_{(p)}[\xi]$-module, where $\xi$ is a primitive $p$-root of unity. Since $\Z[\xi]$ is a Dedekind domain, $M/M^{H_{max}}$ is projective. Thus there is a $\Z[\xi]$-module $M''$ such that $$M\cong_{\Z[\xi]} M^{H_{max}}\oplus M''.$$ 
Notice that the previous isomorphism can be naturally lifted to an isomorphism of $\Z[\Z/p]$-modules.
Now apply Lemma \ref{virtual:maschke} with $\Z[\Z/p]$ as coefficient ring and $G_p$ as a finite group. Thus there is a $\Z[\Z/p][G_p]=\Z[G]$-module $M'$ such that $$|G_p|\cdot M\cong_{\Z[G]} |G_p|\cdot M^{H_{max}}\oplus M'.$$ Recall that $|G_p|$ is coprime with $p$. Define $M_{H_{max}}=|G_p|\cdot M^{H_{max}}$. Since  $M_{H_{max}}$ is nontrivial,  the $\Z$-rank of $M'$ is strictly less than  the $\Z$-rank of $M$. Now the result follows by induction.
\end{proof}

\subsection{Modules of type $(r,s,t)$}\label{section:rst}

We consider $G$, $p$, and $G_p$ as in \cref{notation}. 

\begin{definition}\label{defrst}
    Let $M$ be a $\Z[\Z/p]$-module. We say that $M$ is of type $(r,s,t)$ whenever it is isomorphic as  $\Z[\Z/p]$-module to $$\Z^r\oplus \Z[\Z/p]^s\oplus I^t,$$ where    $I\subseteq\Z[\Z/p]$ the augmentation ideal, and $\Z^r$ is endowed with the trivial $\Z/p$-action.
\end{definition}

\begin{remark}
    Note that, by definition, $\Z^r$ is endowed with the trivial $\Z/p$-action. On the other hand, the $\Z/p$ action on $I^t$ is free outside the origin. In fact, every element of $\Z[\Z/p]$ has either isotropy group trivial or equal to $\Z/p$, the latter elements are exactly the scalar multiples of the norm element. We will be using these facts repeatedly in the rest of the paper. 
\end{remark}

\begin{lemma}\label{lem:existence:uniqueness:rst}
     Let $M$ be a $\Z[\Z/p]$-module. Then the following holds.
     \begin{enumerate}
         \item The module $M$ contains a submodule of finite, coprime with $p$, index of type $(r,s,t)$.
         \item If $M$ cointains two submodules of type $(r,s,t)$ and $(r',s',t')$, then $(r,s,t)=(r',s',t')$.
     \end{enumerate}
\end{lemma}
\begin{proof}
    The first statement is pointed out in \cite[p. 347]{AGPP}. The second statement follows from \cref{remark:uniqueness:rst}.
\end{proof}

\begin{proposition}[Existence of $G$-equivariant $(r,s,t)$-decompositions]\label{prop:virtual:rst:decompositions}
Let $M$ be a $\Z[G]$-module. Assume $M$ is finitely generated and free considered as an abelian group. Then there exist a finite index $\Z[G]$-submodule $N$ of $M$ such that such an index is coprime with $p$, and  a $G$-invariant $(r,s,t)$-decomposition for $N$, i.e.  $N$ can be decomposed as
$$N \cong \Z^r\oplus \Z[\Z/p]^s\oplus I^t$$
such that each factor is a $\Z[G]$-module.
\end{proposition}
\begin{proof}
By \cref{lem:existence:uniqueness:rst}(1), there exists a finite index $\Z[G]$-submodule $M'$ of $M$ such that such an index is coprime with $p$, and $M'$ is of type $(r,s,t)$ as a $\Z[\Z/p]$-module.  Consider any $(r,s,t)$-decomposition $\Z^r\oplus \Z[\Z/p]^s\oplus I^t$ of $M'$. Every element of $M'$, via this decomposition can be expressed as a vector $(a,b,c)$, we will use this notation through the proof.

First we will prove that $|G_p|\cdot\Z^r$ is $G$-invariant. Let $g\in G$, then $g$ defines an automorphism $g:M'\to M'$. Take $(x,0,0)$ in $M$ such that $x$ is an element in the canonical basis in $\Z^r$. Denote $g(x,0,0)=(y,z,w)$, it is enough to prove that $z=0$ and $w=0$. Note that the $\Z/p$-stabilizer of $(x,0,0)$ is $\Z/p$, whilst the $\Z/p$-stabilizer of $w$ is trivial as the $\Z/p$-action on $I$ is free. We conclude that $w=0$. Now suppose that $z\neq 0$ in order to get a contradiction. Let $C$ and $D$ be the $\Z[\Z/p]$-modules generated by $(x,0,0)$ and $g(x,0,0)$, respectively, in particular $D=gC$ and both are isomorphic to $\Z$. Since $z\neq 0$ we get $ D\cap \Z^r=0$. On the other hand, $g$ induces an isomorphism on quotients
 $$\widetilde{g}:M'/C \to M'/D.$$
By the choice we made of $x$, $M'/C$ is clearly of type $(r-1,s,t)$. Finally notice that, since $ D\cap \Z^r=0$,  $M'/D$ contains a copy of $\Z^r$, which  forces $M'/D$ to be of type $(r',s',t')$ with $r'\geq r$. This contradicts the uniqueness of the $(r,s,t)$-decomposition, see \Cref{lem:existence:uniqueness:rst}.

\vskip 10pt

Denote $G_p$ the quotient of $G$ by $\Z/p$. By Lemma \ref{virtual:maschke}, there is a  $\Z[\Z/p][G_p]$-module $N$  such that    $|G_p|\cdot M'\cong \Z^r \oplus |G_p|\cdot N$. Since we have isomorphisms of rings $\Z[\Z/p][G_p]\cong [\Z/p\times G_p]\cong \Z[G]$, we conclude that the aforementioned isomorphism of modules is actually an isomorphism of $\Z[G]$-modules. By \cref{lem:existence:uniqueness:rst}, we have that $N$ is of type $(0,s,t)$, hence $N$ is isomorphic as a  $\Z[\Z/p]$-module to $\Z[\Z/p]^s\oplus I^t$.

We will prove that, in the decomposition of $N$,  $|G_p|\cdot I^t$ is $G$-invariant. The argument is similar as for the $\Z^r$-factor. Let $x\in I^t$ be a generator of one of the $I$-factors, then $g(0,x)=(y,z)\in \Z[\Z/p]^s\oplus I^t$. Suppose that $y\neq 0$. Let $C$ be the $\Z[\Z/p]$-submodule generated by $g(0,x)$, we claim that $C\cap I^t=0$. In fact this intersection is a proper submodule of $C$ and therefore it has $\Z$-rank strictly less than $p-1$, therefore the $(r,s,t)$ decomposition of any finite index submodule can only be of the form $(r,0,0)$, but actually $r$ must be $0$ since $C$ has no fixed points under the $\Z/p$-action.
Taking quotients, we have an isomorphism $$\widetilde{g}: \Z[\Z/p]^s\oplus I^t/C\to  \Z[\Z/p]^s\oplus I^t/gC$$ but the left hand side is a module of type $(0,s,t-1)$ and  the right hand side contains a submodule of finite index with at least $t$ copies of $I$ as direct summand, then in the $(r',s',t')$ decomposition, $t'\geq t$, which is a contradiction. Therefore $y=0$. Finally again by Lemma \ref{virtual:maschke}, $|G_p|\cdot I^t$ has complement $N'$ in $N$ as $\Z[\Z/p][G_p]$-module, which by the uniqueness of the $(r,s,t)$ decomposition must contain a submodule of finite index  of type $(0,s,0)$. This finishes the proof.
\end{proof}

\bibliographystyle{alpha} %harvard, unsrt, alpha
\bibliography{myblib}
\end{document}